\documentclass[10pt]{amsart}

\usepackage{amssymb,amsmath,amsthm,amsfonts}
\usepackage{pdfsync}

\newcommand{\comment}[1]{}

\newtheorem{lem}{Lemma}
\newtheorem{propn}{Proposition}

\newtheorem{thm}{Theorem}

\theoremstyle{remark}
\newtheorem*{Rem}{Remark}
\theoremstyle{definition}

\newcommand{\R}{\mathbb R}

\newcommand{\Z}{\mathbb Z}

\newcommand{\N}{\mathbb N}

\DeclareMathOperator{\lcm}{lcm}

\newcommand{\D}{\delta}

\newcommand{\VE}{\varepsilon}
\newcommand{\A}{\alpha}
\newcommand{\B}{\beta}
\newcommand{\lm}{\lambda}

\newcommand{\be}{\begin{equation}}
\newcommand{\ee}{\end{equation}}
\newcommand{\bee}{\begin{equation*}}
\newcommand{\eee}{\end{equation*}}

\begin{document}
\title{A new proof of S\'ark\"ozy's theorem}
\author{Neil Lyall}

\address{Department of Mathematics, The University of Georgia, Athens, GA 30602, USA}
\email{lyall@math.uga.edu}


\begin{abstract}
It is a striking and elegant fact (proved independently by Furstenberg and S\'ark\"ozy) that in any subset of the natural numbers of positive upper density there necessarily exist two distinct elements whose difference is given by a perfect square. In this article we present a new and simple proof of this result by adapting an argument originally developed by Croot and Sisask to give a new proof of Roth's theorem. 
\end{abstract}
\maketitle

\setlength{\parskip}{3pt}


\begin{center}
\emph{Dedicated to Steve Wainger on the occasion of his retirement}
\end{center}

\section{Introduction}

Let $D(N)$ denote the maximum size of a subset of $\{1,\dots,N\}$ that contains no perfect (non-zero) square differences. In other words, $D(N)$ is the threshold such that if $A\subseteq\{1,\dots,N\}$ with $|A|>D(N)$, then the set $A$ will necessarily contain two distinct elements whose difference if a perfect square.

  In this note we shall be concerned with the behavior of this quantity for large values of $N$ and at the outset we encourage the reader to convince herself of the essentially trivial upper and lower bounds for $D(N)$ of approximate quality $N/4$ and $\sqrt{N}$ respectively, and furthermore that any improvements on these bounds would be less than trivial to achieve. In Appendix \ref{trivial} we give full justification for the following specific bounds
\be\label{start}\sqrt{N}-1\leq D(N)\leq (N+543443)/4.\ee

It was conjecture by Lov\'asz that $D(N)\leq\D N$ for any $\D>0$, provided that $N$ is sufficiently large, or equivalently that in any subset of the natural numbers of positive upper density\footnote{\ Recall that $A\subseteq\N$ is said to have positive upper density whenever
$\limsup_{N\rightarrow\infty}|A\cap\{1,\dots,N\}|/N>0.$} there necessarily exist two distinct elements (and hence infinitely many pairs of distinct elements) whose difference is given by a perfect square. This conjecture was subsequently proven to be correct, independently, by S\'ark\"ozy and Furstenberg.
\begin{thm}[S\'ark\"ozy \cite{Sarkozy}/Furstenberg \cite{Fur}]\label{1}
\[\lim_{N\rightarrow\infty}\frac{D(N)}{N}=0\]
\end{thm}

The purpose of this note is to give a new and simple proof of this result by adapting an argument that was originally developed by Croot and Sisask \cite{CS} to give a new proof of Roth's theorem on three term arithmetic progressions. 
In particular we will establish the following result, which clearly implies Theorem \ref{1}.

\begin{thm}\label{2} Let $M,N\in\N$, then
\[\frac{D(N)}{N}\leq\frac{3}{4}\frac{D(M)}{M}\]
provided $N\geq e^{CM^7}$, for some absolute constant $C>0$, and $M$ is sufficiently large.
\end{thm}

\subsubsection*{Remark on quantitative bounds}
Although its proof is simple, Theorem \ref{2} patently leads to quantitative upper bounds of the quality $N/\log_*N$ for $D(N)$ that are extremely weak\footnote{ \ Recall that $\log_*N$ is the height required for a tower of 2's to exceed $N$.} in comparison to the current best known upper bound, namely
\be
D(N)\leq CN/(\log N)^{\frac{1}{4}\log\log\log\log N}
\ee
for some absolute constant $C>0$, which was established by Pintz, Steiger and Szemer\'edi in \cite{PSS} using an ingenious and intricate Fourier analytic argument. 
For extremely readable accounts of easier arguments leading to intermediate bounds of the quality $N/(\log\log N)^{1/11}$ and $N\log\log N/\log N$, see Green \cite{Green} and Lyall and Magyar \cite{neil}
, respectively.

We further note that it is conjectured that $D(N)\geq N^{1-\VE}$ for any $\VE>0$, provided $N$ is sufficiently large (with respect to $\VE$), and that Ruzsa \cite{Ruzsa} has demonstrated this conjecture to be true for all $\VE\geq0.267$.  

\comment{
The secondary objective of this note is to demonstrate a second proof of Theorem \ref{1}, a simplification of a more general argument due to the author and Magyar \cite{LM1}, that leads to the following intermediate bound.

\begin{thm}[Lyall and Magyar \cite{LM1'}, see also \cite{LM1}]\label{Thm4}
There exists an absolute constant $C>0$ such that
\[\frac{D(N)}{N}\leq C\,\frac{\log\log N}{\log N}.\]
\end{thm}
The proof of this result will be presented in Section \ref{ProofThm4} below.
}

\subsubsection*{Remark on other polynomial differences}
At this point the reader is presumably curious to know what is so special about square differences. The following Theorem gives a complete answer to this question.

\begin{thm}[Kamae and Mend\`es France \cite{KMF}]\label{Thm3}
Let $f\in\Z[n]$ and $D_f(N)$ denote the maximum size of a subset of $\{1,\dots,N\}$ that contains no two distinct elements whose difference is given by $f(n)$ for some $n\in\Z$. Then,
\be\label{KM}\lim_{N\rightarrow\infty}\frac{D_f(N)}{N}=0\ee
if and only if $f$ is an \emph{intersective} polynomial, namely if $f$ has a root modulo $q$ for every $q\geq2$.
\end{thm}

The approach of Kamae and Mend\`es France in \cite{KMF} was indirect and gave no quantitative bounds for $D_f(N)$, and
while the methods of Pintz, Steiger and Szemer\'edi were later extended by Balog, Pelik\'an, Pintz and Szemer\'edi \cite{BPPS} to establish the quantitative bounds
\be\label{BPSS}
D_{n^k}(N)\leq C_kN/(\log N)^{c\log\log\log\log N}\ee
for any integer $k\geq 2$, the current best known upper bounds for general intersective polynomials $f\in\Z[n]$ are due to Lucier \cite{Lucier}, who showed that
\be\label{luc}D_f(N)\leq C_fN\left(\frac{(\log\log N)^\mu}{\log N}\right)^{1/(k-1)}\ee
where $k=\deg(f)$ and $\mu=3$ if $k=2$ and $\mu=2$ if $k\geq3$. However, bounds of the same quality as (\ref{BPSS}) have recently been obtained for general intersective \emph{quadratic} polynomials by Hamel, Lyall and Rice in \cite{HLR}.

The methods used to prove Theorem \ref{2} can in fact be extended, using some (rather technical) additional results of Lucier, to also establish Theorem \ref{Thm3}, these arguments will appear elsewhere. 

A result almost as general as Theorem \ref{Thm3}, namely that
(\ref{KM}) holds whenever $f$ is a polynomial in $\Z[n]$ with at least one integer root\footnote{ \ While it is clear that any polynomial $f$ in $\Z[n]$ with an integer root is plainly intersective, there do in fact exist polynomials with no rational roots that also have this property, for example $(n^3-19)(n^2+n+1)$.}, follows in a more straightforward manner using the same methods as those used in the proof of Theorem \ref{2}, see \cite{SS} (a preliminary version of this current paper) for a brief outline of how to extend the proof of Theorem \ref{2} in this direction. For the current best known upper bounds for this class of polynomials see \cite{LM1} and \cite{LM1'}.

In this note we shall focus exclusively on the case of square differences and proving Theorem \ref{2}. 


\section{Proof of Theorem \ref{2}}\label{ProofThm2}

Let $A\subseteq\{1,\dots,N\}$ with no square differences and $|A|=D(N)$.
Key to the argument we present is to construct, from this extremal set $A$, a new set $B\subseteq\{1,\dots,N\}$ with the following properties:
\begin{itemize}
\item[(i)] $|B|\geq\dfrac{5}{3}|A|$
\item[(ii)] \# of square differences in $B$ $\leq\dfrac{C_0}{\sqrt{\log N}}N^{3/2}$,
for some absolute constant $C_0>0$.
\end{itemize}

This construction, which will amount to defining $B$ to be $A\cup(A+t^2)$ for some appropriate (large) value of $t$, will be carried out in Section \ref{Construction} below.
Having constructed a set with such properties we will then establish Theorem \ref{2} by combining this with the following lower bound on the number of square differences contained in any given set $B\subseteq\{1,\dots,N\}$.

\begin{lem}\label{V}
Given any $B\subseteq\{1,\dots,N\}$ and $1\leq M\leq N$ 
\[\text{\# of square differences in $B$}\geq \left(\frac{|B|/N-(D(M)+2)/M}{M^{5/2}}\right) N^{3/2}.\]
\end{lem}

The proof of this result is a straightforward exercise using ideas that where first exploited by Varnavides \cite{V} in the context of counting three term arithmetic progressions. While, in our context of counting square differences, this quantitative result can easily be deduced by adapting the proof of Theorem 3.1 in \cite{HL} (for example) we will,
for the sake of completeness, include a proof of Lemma \ref{V} in Section \ref{ProofV} below.

We should also note at this point that the standard application Varnavides' argument is to show that Theorem \ref{1}
is equivalent to the statement that for any $\D>0$ and $B\subseteq\{1,\dots,N\}$ with $|B|\geq\D N$ 
\[\text{\# of square differences in $B$}\geq  c(\D)N^{3/2},\] 
for some $c(\D)>0$.
In other words, provided $N$ is sufficiently large, $B$ will contain not only one square difference, but a positive proportion of all the square differences in $\{1,\dots,N\}$. This result clearly follows easily from  Lemma \ref{V}.

\subsection{Proof of Theorem \ref{2}}
It follows immediately from the upper bound on the number of square differences in $B$ given by property (ii) and the lower bound given by Lemma \ref{V}, that
\[\frac{|B|}{N}\leq \frac{D(M)}{M}+\frac{2}{M}+\frac{C_0M^{5/2}}{\sqrt{\log N}}.\]

Assuming that $N$ satisfies $C_0M^{7/2}\leq\sqrt{\log N}$, it follows that
\[
\frac{|B|}{N}\leq \frac{D(M)}{M}+\frac{3}{M}
\]
and hence, using the trivial lower bound $D(M)\geq\sqrt{M}-1$ (see Section \ref{a2}), that
\[\frac{|B|}{N}\leq\frac{5}{4}\frac{D(M)}{M}\]
provided that $M$ is sufficiently large (in fact $M\geq169$ is sufficient).
Combining this observation with the inequality
\[
\frac{|B|}{N}\geq\frac{5}{3}\frac{|A|}{N}=\frac{5}{3}\frac{D(N)}{N}
\]
which follows immediately from property (i) of our constructed set $B$, gives the desired inequality.\qed

\subsection{Construction of the set $B$}\label{Construction}
Given any set $B\subseteq\{1,\dots,N\}$, it is easy to see that
\be\label{3}
\text{\# of square differences in $B$}=\sum_{n=1}^{\sqrt{N}}\sum_{x\in\Z}B(x)B(x-n^2)
\ee
where $B(x)=1_B(x)$ denotes the indicator function of the set $B$.
Using the familiar orthogonality relation 
\[\int_0^1e^{2\pi i x \A}d\A=\begin{cases}
1\quad\text{if \ $x=0$}\\ 
0\quad\text{if \ $x\in\Z\setminus\{0\}$}
\end{cases}\]
we can, as is standard, express our count (\ref{3}) on the ``transform side'' as
\be\label{4}
\text{\# of square differences in $B$}=\int_0^1|\widehat{B}(\A)|^2\widehat{S}(\A)\,d\A
\ee
where
\[\widehat{B}(\A)=\sum_{x\in\Z}B(x)e^{-2\pi i x\A}\]
denotes the Fourier transform (on $\Z$) of the set $B$ and
\be
\widehat{S}(\A)=\sum_{n=1}^{\sqrt{N}}e^{-2\pi i n^2\A}
\ee
is the Fourier transform of the set of perfect squares contained in $\{1,\dots,N\}$.

Key to our proof (and essentially the only true ``machinary" used in the proof) is the following well-known estimate for the Weyl sum $\widehat{S}(\A)$, which states that the only possible obstruction to cancellation in this exponential sum arises if $\A$ is ``close'' to a rational with ``small'' denominator. 
\begin{propn}\label{Weyl Estimates}
Let $\VE>0$ and 
\[
\mathbf{M}_{a/q}(\VE)=\left\{\A\in[0,1]\,:\,\Bigl|\A-\frac{a}{q}\Bigr|\leq\frac{1}{\VE^{2} N}\right\}.
\]
If $\A\notin\mathbf{M}_{a/q}(\VE)$ for any $(a,q)=1$ with $1\leq q\leq\VE^{-2}$, then
\[
|\widehat{S}(\A)|\leq 5\VE N^{1/2}\]
provided $N$ is sufficiently large with respect to $\VE$, in particular $N\geq C\VE^{-50}$ would be sufficient.
\end{propn}
\comment{
\begin{Rem}[on ``big O notation''] Whenever we write
$E=O(F)$ for any two quantities $E$ and $F$ we shall mean that $|E|\leq CF$, for some constant $C>0$.
\end{Rem}}

We are now ready to define our set $B$. Recalling that $A\subseteq\{1,\dots,N\}$ is an extremal set with no square differences, we define (for a value of $\VE>0$ to de determined)
\be
B:=A'\cup(A'+q_\VE^2)
\ee
where $q_\VE=\lcm\{1\leq q \leq \VE^{-2}\}$ and $A'=A\cap\{1,\dots,N-q_\VE^2\}$.

Using the fact that $|A|=D(N)\geq\sqrt{N}-1$ it follows that $|A'|\geq5|A|/6$, and consequently also that property (i) for our set $B$ will hold, provided $\VE>0$ is chosen large enough for
\be\label{7}
q_\VE^2\ll \sqrt{N}.
\ee
In order to see what actual restriction this places on our choice of $\VE>0$, we recall, as one can verify using only elementary properties of the prime numbers, that \[\exp(\VE^{-2}/2)\leq q_\VE\leq\exp(\VE^{-2})\] and hence that inequality (\ref{7}) will hold whenever \[\VE^{-2}\ll \log N.\] 

\begin{Rem}[on ``$\ll$ notation''] Whenever we write
$E\ll F$ for any two quantities $E$ and $F$ we shall mean that $E\leq cF$, for some some sufficiently small constant $c>0$.
\end{Rem}

We therefore now fix
\be
\VE:=C_1(\log N)^{-1/2}
\ee 
with $C_1>0$ a sufficiently large (but absolute) constant.
In order to to establish that our set $B$ also satisfies property (ii) it will suffice to show, for this choice of $\VE>0$, that
\be\label{9}
\text{\# of square differences in $B$}\leq 20\,\VE N^{3/2}
\ee 
for all sufficiently large $N$.

To establish (\ref{9}) we first note that since $A'\subseteq A$ contains no square differences, it follows that 
\[B(x)=A'(x)+A'(x-q_\VE^2)\]
since $A'$ and $A'+q_\VE^2$ are disjoint, and hence, using the familiar and easily verified property that Fourier transformation takes translations to modulations, that
\[
\widehat{B}(\A)=\widehat{A'}(\A)(1+e^{-2\pi i q_\VE^2\A}).
\]
Multiplying this expression for $\widehat{B}(\A)$ by its complex conjugate, we see that
\[
\int_0^1|\widehat{B}(\A)|^2\widehat{S}(\A)\,d\A=2\int_0^1|\widehat{A'}(\A)|^2(\cos(2\pi q_\VE^2\A)+1)\widehat{S}(\A)\,d\A.
\]

In light of (\ref{4}), and the fact that $A'$ contains no square differences, it follows that
\[\int_0^1|\widehat{A'}(\A)|^2\widehat{S}(\A)\,d\A=0\]
and hence that
\begin{align*}
\text{\# of square differences in $B$}&=2\int_0^1|\widehat{A'}(\A)|^2(\cos(2\pi q_\VE^2\A)-1)\widehat{S}(\A)\,d\A\\
&\leq 2\int_0^1|\widehat{A'}(\A)|^2\underbrace{|\cos(2\pi q_\VE^2\A)-1||\widehat{S}(\A)|}_{(\star)}\,d\A.
\end{align*}

A crucial observation at this point, which completes the proof of inequality (\ref{9}), is the fact that
\be\label{crucial}
(\star)\leq 10\,\VE \sqrt{N}
\ee 
uniformly in $\A$. It then follows that
\[
\text{\# of square differences in $B$}\leq 20\,\VE \sqrt{N}\int_0^1|\widehat{A'}(\A)|^2\,d\A\leq 20\,\VE N^{3/2}
\]
where to establish the final inequality we have invoked the Plancherel identity, namely
\[\int_0^1|\widehat{A'}(\A)|^2\,d\A=\sum_{x\in\Z}|A'(x)|^2\]
whose validity in this setting can be easily verified (using orthogonality), together with the simple observation that 
\[\sum_{x\in\Z}|A'(x)|^2=|A'|\leq N.\]

It remains to verify the uniform estimate (\ref{crucial}). Since $|\cos(2\pi q_\VE^2\A)-1|\leq 2$ for all $\A\in [0,1]$, it follows from Proposition \ref{Weyl Estimates} that (\ref{crucial}) will hold whenever $\A\notin\mathbf{M}_{a/q}(\VE)$ for any $(a,q)=1$ with $1\leq q\leq\VE^{-2}$, since $N=\exp(C_1^2\VE^{-2})\gg \VE^{-50}$. While if $\A\in\mathbf{M}_{a/q}(\VE)$ for some $(a,q)=1$ with $1\leq q\leq\VE^{-2}$, then by definition we know that
$|\A-a/q|\leq\VE^{-2} N^{-1}$.
Moreover, since $q|q_\VE^2$ (by the definition of $q_\VE$) it follows that
\[\cos\left(2\pi q_\VE^2\A\right)=\cos\left(2\pi q_\VE^2\left(\A-a/q\right)\right)\]
and hence, by the Mean Value Theorem, we see that
\begin{align*}
|\cos(2\pi q_\VE^2\A)-1|&=|\cos(2\pi q_\VE^2\left(\A-a/q\right))-1|\\
&\leq2\pi q_\VE^2|\A-a/q|\\
&\leq 2\pi q_\VE^2\VE^{-2} N^{-1}.
\end{align*}
The result then follows, provided the constant $C_1$ in our choice of fixed $\VE>0$ is chosen sufficiently large, since
\[2\pi q_\VE^2\VE^{-2} N^{-1}\leq\VE\]
whenever $\VE^{-2}\ll\log N$ (again) and we trivially know that $|\widehat{S}(\A)|\leq\sqrt{N}$ for all $\A\in[0,1]$.
\qed

This completes the proof of Theorem \ref{2} modulo Lemma \ref{V} and Proposition \ref{Weyl Estimates}. The proof of these two results are given in Section \ref{ProofVandW} below.

\section{Proof of Lemma \ref{V} and Proposition \ref{Weyl Estimates}}\label{ProofVandW}

\subsection{Proof of Lemma \ref{V}}\label{ProofV}

Let $B\subseteq\{1,\dots,N\}$ and $1\leq M\leq N$.
We proceed by covering \{1,\dots,N\} by the collection of all square-difference progressions of length $M$ of the form
\[P_{a,r}=\{a+r^2,\dots,a+Mr^2\}\]
with $1\leq r\leq R:=\sqrt{N}/M$ and $1\leq a\leq N-MR^2$. 
We will say that such a progression $P_{a,r}$ is \emph{good} if
\[|B\cap P_{a,r}|\geq D(M)+1\]
since, by virtue of the fact that square differences are preserved under translations and dilations by a perfect square, each such progression clearly contributes at least one square difference in $B$.

A simple counting argument, which we give below, shows that 
\be\label{count}\text{\# of \emph{good} progressions $P_{a,r}$}\geq\left(\frac{|B|}{N}-\frac{D(M)+2}{M}\right) RN.\ee
Now while, as noted above, each of these good progressions contributes at least one square difference in $B$, it is of course also the case that some of these square differences could be getting over counted. However, as we shall also see below, each square difference in $B$ is being over 
counted at most $M^{3/2}$ times, from which it follows that
\[\text{\# of square differences in $B$}\geq \left(\frac{|B|/N-(D(M)+2)/M}{M^{5/2}}\right) N^{3/2}\]
as required.
We are thus left with the straightforward tasks of verifying (\ref{count}) and the claim that the each square difference in $B$ is being over 
counted in this argument at most $M^{3/2}$ times. 

We will address the over counting argument first.
Suppose we are 
given a pair $\{b,b+n^2\}$ in $B$. If this pair is contained in $P_{a,r}$, then $r$ must be a divisor of $n$ and moreover $n^2\leq Mr^2$. It therefore follows that there are at most $\sqrt{M}$ choices for $r$ and it is easy to see that each choice of $r$ fixes $a$ in at most $M$ ways, thus each square difference is  indeed over counted at most $M^{3/2}$ times. 

Finally, we verify (\ref{count}). By combining the upper bound
\[
\sum_{r=1}^R\sum_{a=1}^{N-MR^2}|B\cap P_{a,r}|\leq \!\!\!\!\sum_{\substack{a,r \\ \text{good }P_{a,r}}}\!\!\!\!M\,\,\,+\!\!\!\!\!\!\!\!\sum_{\substack{a,r \\ \text{\underline{not} good }P_{a,r}}}\!\!\!\!\!\!\!\!D(M)\leq (\text{\# of \emph{good} progressions $P_{a,r}$})\,M + D(M)RN
\]
with the lower bound
\[
\sum_{r=1} ^R\sum_{a=1}^{N-MR^2}|B\cap P_{a,r}|\geq M\sum_{r=1} ^R |B\cap\{Mr^2,\dots, N-Mr^2\}|\geq MR\left(|B|-2MR^2\right)
\]
it follows that
\[\text{\# of \emph{good} progressions $P_{a,r}$}\geq \left(\frac{|B|}{N}-\frac{2MR^2}{N}-\frac{D(M)}{M}\right) RN\]
from which (\ref{count}) follows.\qed


\subsection{Proof of Proposition \ref{Weyl Estimates}}

We first recall Dirichlet's (pigeonhole) principle: 

\begin{center}
\emph{Given any $\A\in\R$ and $Q\in\N$, there exist $(a,q)=1$ with $1\leq q\leq Q$ such that} 
\[\Bigl|\A-\frac{a}{q}\Bigr|\leq\frac{1}{qQ}\leq\min\Bigl\{\frac{1}{q^2},\frac{1}{Q}\Bigr\}.\]
\end{center}

The proof of the following key result is completely standard, see for example \cite{TenLectures} or \cite{gowersnotes}.

\begin{propn}[The Weyl inequality]\label{Weyl}

If $|\A-a/q|\leq q^{-2}$ and $(a,q)=1$, then
\[|\widehat{S}(\A)|\leq 40\sqrt{N}\log N(1/q+1/\sqrt{N}+q/N)^{1/2}.\]
\end{propn}

We note (by Dirichlet's principle) that for any given $\A\in\R$ and $Q\in\N$, there always exist $(a,q)=1$ with $1\leq q\leq Q$ that satisfies the hypothesis of the Weyl inequality. Moreover, it is easy to see that this inequality gives a non-trivial conclusion whenever $N^{\mu}\leq q\leq N^{1-\mu}$ for some $0<\mu<1/2$. 
For the purposes of this exposition we shall take $Q=N^{1-\mu}$ with $\mu=1/20$ and define
\[
\mathbf{M}_{a/q}'=\left\{\A\in [0,1]\,:\,\Bigl|\A-\frac{a}{q}\Bigr|\leq \frac{1}{N^{19/20}} \right\}.
\]

It is customary to say that $\A$ is in a \emph{major arc} if $\A\in \mathbf{M}_{a/q}'$ for some $(a,q)=1$ with $1\leq q \leq N^{1/20}$, and call the complement of these major arcs, the \emph{minor arcs}.
If $\A$ is in one of these minor arcs, then it follows from Dirichlet's principle that there must exist a reduced fraction $a/q$ with $N^{1/20}\leq q\leq N^{19/20}$ such that $|\A-a/q|\leq q^{-2}$, and hence, by the Weyl inequality, that
\[|\widehat{S}(\A)|\leq 80 N^{19/40}\log N\leq \VE \sqrt{N}\]
for any $\VE>0$ that satisfies $N\gg\VE^{-50}$. 

In order to obtain the full conclusion of Proposition \ref{Weyl Estimates}, which is valid on a subset of $[0,1]$ which is strictly larger than the collection of classical minor arcs defined above, we must perform a careful analysis of the behavior our exponential sum $\widehat{S}(\A)$ on the major arcs. In particular, we will invoke the following.

\begin{lem}[Major arc estimate]\label{major}
If $\A\in \mathbf{M}_{a/q}'$ for some $(a,q)=1$ with $1\leq q\leq N^{1/20}$, then
\[
|\widehat{S}(\A)|\leq 5\sqrt{N}\,q^{-1/2}(1+N|\A-a/q|)^{-1/2}.
\]
\end{lem}

It now follows immediately from this Lemma that for any given $\VE>0$ and $\A\in\mathbf{M}_{a/q}'$, our exponential sum will satisfy
\[|\widehat{S}(\A)|\leq 5\VE \sqrt{N}\]
provided $(a,q)=1$ and either 
$\VE^{-2}\leq q\leq N^{1/20}$
or 
$\VE^{-2}N^{-1}\leq |\A-a/q|\leq N^{-19/20},$ as required.\qed

The proof of Lemma \ref{major} is standard, but for the sake of completeness we have chosen to included a proof in Appendix \ref{A3} below. 

\appendix

\section{Proof of Lemma \ref{major} (Major arc estimate)}\label{A3}

\comment{Before launching into this, we make the important observation that the major arcs (and hence the refined major arcs) are a union of (necessarily short) pairwise disjoint intervals.
\begin{lem}
If $a/q\ne a'/q'$ with $1\leq q,q'\leq N^{1/20}$, then $\mathbf{M}'_{a/q}\cap \mathbf{M}'_{a'/q'}=\emptyset$.
\end{lem}
\begin{proof}
Suppose that $\mathbf{M}'_{a/q}\cap \mathbf{M}'_{a'/q'}\ne\emptyset$. Using the fact that $aq'-a'q\ne0$, we see that
\[\frac{2}{N^{1-1/20}}\geq\Bigl|\frac{a}{q}-\frac{a'}{q'}\Bigr|=\Bigl|\frac{aq'-a'q}{qq'}\Bigr|\geq \frac{1}{qq'}\geq\frac{1}{N^{1/10}},\]
a contradiction.
\end{proof}

\begin{proof}[Proof of Lemma \ref{minorI}]
It follows from the Dirichlet principle and the fact that $\A$ is in a minor arc that there exists a reduced fraction $a/q$ with \[N^{1/20}\leq q\leq N^{1-1/20}\] such that $|\A-a/q|\leq q^{-2}$. It therefore follows from the Weyl inequality that
\[|\widehat{S}(\A)|\leq 50 N^{1/2-1/40}\log N\leq C N^{1/2-1/80}.\qedhere\]
\end{proof}
}

The proof of Lemma \ref{major} hinges on the key observation that for each $\A$ in a major arc corresponding to a rational $a/q$, our exponential sum $\widehat{S}(\A)$ breaks naturally into an arithmetic part $S(a,q)$ and a continuous part $I_N(\A-a/q)$,
up to a manageable error term. In particular we have

\begin{lem}\label{approxpropn}
If $\A\in\mathbf{M}'_{a/q}$ with $1\leq q\leq N^{1/20}$, then
\be\label{approx}
\widehat{S}(\A)=\sqrt{N}\,q^{-1}S(a,q)I_N(\A-a/q)+O(N^{1/10})
\ee
where
\[S(a,q)=\sum_{r=0}^{q-1}e^{-2\pi iar^2/q}\quad\text{and}\quad I_N(\B)=\int_0^{1} e^{-2\pi i N\B x^2}dx.\]

\end{lem}

\begin{Rem}[on ``big O notation''] Whenever we write
$E=O(F)$ for any two quantities $E$ and $F$ we shall mean that $|E|\leq CF$, for some constant $C>0$.
\end{Rem}

\begin{proof}
We can write $\A=a/q+\B$ where $|\B|\leq 1/N^{19/20}$ and $1\leq q\leq N^{1/20}$. We can also write each $1\leq d\leq \sqrt{N}$ uniquely as $d=m q+r$
with $1\leq r\leq q$ and $0\leq m\leq \sqrt{N}/q$. It then follows that
\begin{align*}
\widehat{S}(\A)&=\sum_{r=1}^{q}\sum_{m=0}^{\sqrt{N}/q} e^{-2\pi i (a/q+\B)(m q+r)^2}+O(q)\\
&=\sum_{r=1}^{q}e^{-2\pi i ar^2/q}\sum_{m=0}^{\sqrt{N}/q} e^{-2\pi i \B(m q+r)^2} +O(q).
\end{align*}
Since 
\begin{align*}
\Bigl|e^{-2\pi i(m q+r)^2\B}-e^{-2\pi i m^2 q^2 \B}\Bigr|&\leq \Bigl|e^{-2\pi i(2m qr+r^2)\B}-1\Bigr|\leq Cdr|\B|\leq CqN^{-9/20}
\end{align*}
and 
\begin{align*}
\Bigl|\sum_{m=0}^{\sqrt{N}/q}e^{-2\pi i m^2 q^2 \B}-\int_0^{\sqrt{N}/q}e^{-2\pi i x^2 q^2 \B}dx\Bigr|&\leq\sum_{m=0}^{\sqrt{N}/q}\int_{m}^{m+1}\Bigl|e^{-2\pi i m^2 q^2 \B}-e^{-2\pi i x^2 q^2 \B}\Bigr|\,dx\\
&\leq \sum_{m=0}^{\sqrt{N}/q} 2\pi(2m+1)q^2|\B|\\
&\leq CN^{1/20}\end{align*}
it follows that
\[\Bigl|\widehat{S}(\A)-\sqrt{N}\,q^{-1}S(a,q)I_N(\B)\Bigr|\leq CN^{1/10}.\qedhere\]
\end{proof}

Lemma \ref{major} follows almost immediately from this and the two basic lemmas below. 
\begin{lem}[Gauss sum estimate]\label{gs}
If $(a,q)=1$, then $|S(a,q)|\leq \sqrt{2q}.$
More precisely,
\[|S(a,q)|=\begin{cases} 
\sqrt{q}\quad&\text{if \ $q$ odd}\\
\sqrt{2q}&\text{if \ $q\equiv0\mod 4$}\\
0&\text{if \ $q\equiv2\mod 4$}
\end{cases}.\]
\end{lem}

\begin{lem}[Oscillatory integral estimate]\label{oi}
For any $\lm\geq0$
\[\Bigl|\int_0^1 e^{2\pi i \lm x^2}dx\Bigr|\leq \min\{1,2\lm^{-1/2}\}\leq 2\sqrt{2}(1+\lm)^{-1/2}.\]
\end{lem}

\begin{proof}[Proof of Lemma \ref{major}]
Lemmas \ref{gs} and \ref{oi} imply that the main term in (\ref{approx})
\[\sqrt{N}\,q^{-1}S(a,q)I_N(\A-a/q)\leq 4\sqrt{N}q^{-1/2}(1+N|\A-a/q|)^{-1/2}\]
and since 
$q^{-1/2}\geq N^{-1/40}$ and $N\,|\A-a/q|\leq N^{1/20}$, it follows that
\[N^{1/10}\ll\sqrt{N}\,q^{-1/2}(1+N|\A-a/q|)^{-1/2}.\qedhere\]
\end{proof}

\begin{proof}[Proof of Lemma \ref{gs}]
Squaring-out $S(a,q)$ we obtain
\[|S(a,q)|^2=\sum_{s=0}^{q-1}\sum_{r=0}^{q-1}e^{2\pi ia(r^2-s^2)/q}.\]
Letting $r=s+t$ and using the fact that $(a,q)=1$ and 
\[\sum_{s=0}^{q-1}e^{2\pi ia(2st)/q}=\begin{cases}
q\quad&\text{if}\quad 2at\equiv 0\mod q\\
0&\text{otherwise}
\end{cases}\]
it follows that
\[|S(a,q)|^2=\sum_{t=0}^{q-1}e^{2\pi iat^2/q}\sum_{s=0}^{q-1}e^{2\pi ia(2st)/q}
=\begin{cases}
q \ &\text{if \ $q$ odd}\\
q\left(e^{2\pi ia(q/4)}+1\right) \ &\text{if \ $q$ even}
\end{cases}.\qedhere\] 
\end{proof}

\begin{proof}[Proof of Lemma \ref{oi}]
We need only consider the case when $\lm\geq1$. We write
\[\int_0^1 e^{2\pi i \lm x^2}dx=\int_0^{\lm^{-1/2}}e^{2\pi i \lm x^2}dx+\int_{\lm^{-1/2}}^1 e^{2\pi i \lm x^2}dx=:I_1+I_2.\]
It is easy to then see that
$|I_1|\leq \lm^{-1/2},$
while integration by parts gives that
\begin{align*}
|I_2|&=\left| \int_{\lm^{-1/2}}^1 \frac{1}{4\pi i\lm x}\Bigl(\frac{d}{dx}e^{2\pi i \lm x^2}\Bigr)dx\right|\\
&\leq\frac{1}{4\pi \lm}\left|\left[\frac{1}{x}\,e^{2\pi i\lm x^2}\right]^1_{\lm^{-1/2}}+\int_{\lm^{-1/2}}^1\frac{1}{x^2}\,e^{2\pi i\lm x^2}dx\right|\\
&\leq \lm^{-1/2}.\qedhere
\end{align*}
\end{proof}

\section{Justification of inequality (\ref{start}): the purported trivial bounds for $D(N)$}\label{trivial}

\subsection{Upper bounds} Let $A\subseteq\{1,\dots,N\}$ with no square differences. 

It clearly follows that
$A\cap(A+t^2)=\emptyset$
for all $t\in\N$ and in particular that
\[|A|\leq (N+1)/2\]
since $|(A+1)\cap\{1,\dots,N\}|\geq |A|-1$ and hence
\[2|A|-1\leq|(A\cup(A+1))\cap\{1,\dots,N\}|\leq N.\]

In order to obtain the superior bound (at least when $N\geq 65$) of \[|A|\leq(N+34)/3,\] one can use the further observation that if $(r,s,t)$ form a Pythagorean triple with $r^2+s^2=t^2$, then
\[A\cap(A+s^2)=A\cap(A+t^2)=(A+s^2)\cap(A+t^2)=\emptyset.\]
In particular, taking $s=3$ and $t=5$, it follows that
\[3|A|-34\leq|\left(A\cup(A+9)\cup(A+25)\right)\cap\{1,\dots,N\}|\leq N,\]
as required, since clearly $|(A+9)\cap\{1,\dots,N\}|\geq |A|-9$ and $|(A+25)\cap\{1,\dots,N\}|\geq |A|-25$.

The superior bound (at least when $N\geq 1630193$) of \[|A|\leq(N+543443)/4\] claimed in the introduction, follows (as above) once one observes that 
\[153^2, 185^2,  697^2, 185^2-153^2, 697^2-185^2, 697^2-153^2\]
are all perfect squares.

\subsection{Lower bound}\label{a2}
We now show that given any subset $H$ of the natural numbers and any $N\in\N$, there always exists a set $A\subseteq\{1,\dots,N\}$ such that
$(A-A)\cap H=\emptyset$
and
\be\label{r}|A|\geq \frac{N-1}{|H\cap\{1,\dots,N\}|+1}.\ee
Taking $H$ to be the set of square numbers, this corresponds to the desired lower bound
$D(N)\geq \sqrt{N}-1.$

We construct the set $A$ 
 recursively as follows: 
Select $a_1=1$ to be the first element in $A$. Having selected $a_1,\dots,a_k$, with $k\geq1$, we define $X_k=\{a_1,\dots,a_k\}+H\cap\{1,\dots,N\}$ and select $a_{k+1}$ to be the smallest element in $\{1,\dots,N\}\setminus\{a_1,\dots,a_k,X_k\}$. In order to guarantee the existence of such an element $a_{k+1}$, we clearly must have $|\{a_1,\dots,a_k,X_k\}|\leq N-1$, and since it is possible that
  $|\{a_1,\dots,a_k,X_k\}|=k\left(|H\cap\{1,\dots,N\}|+1\right)$, this corresponds to the restriction that
  \[k\leq\left\lfloor\frac{N-1}{|H\cap\{1,\dots,N\}|+1}\right\rfloor\]
from which (\ref{r}) immediately follows.


\end{document}